\documentclass[11pt]{amsart}
\usepackage[a4paper,margin=2.0cm]{geometry}

\usepackage{setspace}
\usepackage{amssymb}
\usepackage{amsmath}
\usepackage{amsthm}
\usepackage{array}
\usepackage{xcolor}

\newtheorem{theorem}{Theorem}[section]
\newtheorem{proposition}[theorem]{Proposition}
\newtheorem{lemma}[theorem]{Lemma}

\theoremstyle{definition}
\newtheorem{definition}[theorem]{Definition}

\theoremstyle{remark}
\newtheorem{remark}[theorem]{Remark}

\newcommand{\Image}[1]{{\mathsf{Im}({#1})}}
\newcommand{\trace}[1]{{\mathsf{tr}       \left({#1}\right)}}
\newcommand{\tr}   [2]{{\mathsf{tr}_{{#1}}\left({#2}\right)}}

\newcommand{\DE} [2]{{\mathsf{D}_{#1}^{{#2}}}}
\newcommand{\EA} [0]{{\mathsf{EA}}}
\newcommand{\CCZ}[0]{{\mathsf{CCZ}}}

\newcommand{\DD}[0]{\mathbb{D}}

\newcommand{\DDx}[0]{{\mathbb{D}^\times}}
\newcommand{\FF}[0]{\mathbb{F}}

\newcommand{\LL}[0]{\mathbb{L}}
\newcommand{\LLx}[0]{{\mathbb{L}^\times}}

\newcommand{\f}[1]{\mathbb{F}_{#1}}

\newcommand{\ProjectiveLine}[0]{{\mathbb{P}^1}}
\newcommand{\PLL}[0]{{\mathbb{P}^1(\LL)}}
\newcommand{\Mobius}[0]{{\mathfrak{M}}}
\newcommand{\LLinear}[0]{{\mathfrak{L}}}

\newcommand{\qbar}[0]{{\overline{q}}}
\newcommand{\Hbar}[0]{{\overline{\mathcal{H}}}}
\newcommand{\HH}[0]{{\mathcal{H}}}

\newcommand{\Fset}[0]{\mathcal{F}_{q,\LL}}
\newcommand{\Vset}[0]{\mathcal{V}_{q,\LL}}
\newcommand{\Sset}[0]{\mathcal{S}_{q,\LL}}

\newcommand{\lms}{\{\!\!\{}
\newcommand{\rms}{\}\!\!\}}

\DeclareMathOperator{\wt}{wt}

\DeclareMathOperator{\Gcd}{gcd}
\DeclareMathOperator{\Ker}{Ker}
\DeclareMathOperator{\GL}{GL}
\DeclareMathOperator{\PGL}{PGL}
\DeclareMathOperator{\PSL}{PSL}
\DeclareMathOperator{\mult}{mult}
\DeclareMathOperator{\algdeg}{deg_{alg}}

\newcommand{\Leq}[0] {{\, {\sim_{\mathfrak{L}}}    \,}}
\newcommand{\Peq}[0] {{\, {\sim_{\mathfrak{M}}}    \,}}
\newcommand{\GLEQ}[0]{{\, {\approx_{\mathfrak{L}}} \,}}

\usepackage[pdftex,unicode]{hyperref}   
\hypersetup{unicode}
\hypersetup{breaklinks=true}
\hypersetup{urlcolor=blue}

\begin{document}

\title{Classification of $(q,q)$-biprojective APN functions}%
\author{Faruk G\"{o}lo\u{g}lu}
\address{MFF, Charles University, Prague.}
\curraddr{}
\email{faruk.gologlu@mff.cuni.cz}
\thanks{MFF, Charles University, Prague. ({\tt faruk.gologlu@mff.cuni.cz})\\
This work was supported by the {\sf GA\v{C}R Grant 18-19087S - 301-13/201843}.}

\date{}

\dedicatory{}

\begin{abstract}
In this paper, we classify $(q,q)$-biprojective almost perfect nonlinear (APN) 
functions over $\LL \times \LL$ under the natural left and right action of 
$\GL(2,\LL)$ where $\LL$ is a finite field of characteristic $2$. 
This shows in particular that the only quadratic APN functions (up to $\CCZ$-equivalence)
over $\LL \times \LL$ that satisfy the so-called subfield property 
are the Gold functions and the function $\kappa : \f{64} \to \f{64}$ 
which is the only known APN function that is equivalent to a permutation over 
$\LL \times \LL$ up to $\CCZ$-equivalence. The $\kappa$-function was introduced in
(Browning, Dillon, McQuistan, and Wolfe, 2010). Deciding whether there exist other 
quadratic APN functions (possibly $\CCZ$-equivalent to permutations) that satisfy
subfield property or equivalently, generalizing $\kappa$ to higher dimensions was an
open problem listed for instance in (Carlet, 2015) as one of the interesting open 
problems on cryptographic functions.
\end{abstract}

\maketitle

\section{Introduction}

Almost perfect nonlinear (APN) functions are cryptographically important
functions over a vector space over the finite field of order two that provides 
the best resistance against differential cryptanalysis.
Arguably the most important open problem on APN functions is the
question on the existence of APN permutations over even dimensional vector spaces.
There are no APN permutations over $\f{2}^2$ and $\f{2}^4$ \cite{apnfour}. 
Existence of an APN permutation over $\f{2}^6$ was shown in \cite{jfd10}.
The function is $\CCZ$-equivalent (see Sections \ref{sec_2} and \ref{sec_3} 
for the definitions of the concepts that are used in Introduction) to
the quadratic function $\kappa : \f{2^6} \to \f{2^6}$. Note that we will view the
finite field $\f{2^{2l}}$ as a vector space $\f{2}^{2l}$ and also as 
$\f{2^l}\times \f{2^l}$. When $\kappa$ is viewed as a polynomial over the 
finite field $\f{2^6}$ it falls into the class of Dembowski-Ostrom polynomials
that satisfy the subfield property introduced in \cite{jfd10}. These 
polynomial functions when viewed as functions over $\f{2^l} \times \f{2^l}$
form the class of $(q,q)$-biprojective functions. Of course a natural 
question that arises is to determine whether Dembowski-Ostrom polynomials
that satisfy the subfield property (or, equivalently, the class of $(q,q)$-biprojective 
functions) contain APN functions that are $\CCZ$-equivalent to permutations for $l > 3$.
The question is viewed as an important problem. For instance it was included 
by Carlet in a list of interesting research questions regarding crytographic functions
\cite[Section 3.7]{Carlet15} as an important subproblem of the major problem of deciding 
whether there exists an APN permutation when $l > 3$. 
The problem also appeared in \cite[Problem 15]{G15}. 
In this paper, we solve this problem. 

\begin{theorem}\label{thm_main}
Let $q = 2^k$, $r = 2^l$, $\LL = \f{2^l}$ with $0 < k < l$ and 
$F : \LL \times \LL \to \LL \times \LL$
be a $(q,q)$-biprojective function. Then $F$ is APN if and only if $\Gcd(k,l) = 1$, and
\begin{enumerate}
\item $l$ is even and $F \GLEQ G_{q+1}$ or $F \GLEQ G_{q+r}$, or
\item $l$ is odd, $k$ is odd, and  $F \GLEQ G_{q+1}$, or 
\item $l$ is odd, $k$ is even, and $F \GLEQ G_{q+r}$, or
\item $l = 3$ and $F \GLEQ \kappa$.
\end{enumerate}
\end{theorem}

We have the following clarifications for the statement of the theorem.
\begin{itemize}

\item  The equivalence relation $\GLEQ$ is introduced by 
the action of $\GL(2,\LL) \times \GL(2,\LL)$ which is a finer notion of 
equivalence than that of $\CCZ$-equivalence. Thus our results imply 
$\CCZ$-equivalence results.

\item  The biprojective maps $G_s : \LL \times \LL \to \LL \times \LL$ 
with $s \in \{q+1,q+r\}$ are 
the so-called Gold maps $X \mapsto X^s$ in the univariate notation
with a suitable identification of the vector spaces. 

\item The Gold maps $G_{2^i(2^j+1)}$ over $\LL \times \LL$ are APN whenever 
$\Gcd(j,2l) = 1$ by \cite{gold}. APN Gold maps over $\LL \times \LL$ are not 
$\CCZ$-equivalent to permutations by \cite{GL20}. 
Thus, a $(q,q)$-biprojective APN function $F$ over $\LL \times \LL$ is 
$\CCZ$-equivalent to a permutation if and only if $l = 3$ and $F \GLEQ \kappa$.

\item The special case $k = 1$ was solved in \cite{CL21} using
results from \cite{LLHQ,Dasa,GKL}. Thus, our theorem generalizes the main results
of \cite{CL21,LLHQ,Dasa,GKL}.

\item Another idea to attack the problem is to identify a class of functions
that includes $\kappa$ when $l = 3$ and also are $\CCZ$-equivalent 
to permutations for larger $l$. This is the case of the so-called butterfly 
construction \cite{PUB} which requires $l$ to be odd. In \cite{CPT,CDP}, butterflies
were shown not to be APN when $l > 3$. We will show in Remark \ref{rem_butterfly} that 
a subcase of our Proposition \ref{prop_s} (which is a subcase of the main theorem) 
strictly generalizes the butterfly construction. Thus, our theorem 
also generalizes the main results of \cite{CPT,CDP}.

\item The proof is based on three concepts:
\begin{itemize}
\item zeroes of projective polynomials \cite{Bluher},
\item properties of Dillon-Dobbertin difference sets \cite{DD}, and
\item recent classification of fractional projective permutations over finite fields \cite{FG22FFA}.
\end{itemize}

\item The proof avoids the use of Weil bound and is purely combinatorial.
\end{itemize}

The natural actions of the groups
$\GL(2,\LL) \times \GL(2,\LL)$ 
(left and right application of non-singular $\LL$-linear transformations)
and
($\LLx \times \LLx) \times \GL(2,\LL)$ 
(scaling on both components and right application of non-singular $\LL$-linear transformations) 
on bivariate vectorial Boolean functions of type
\[
F : \LL \times \LL \to \LL \times \LL
\]
and the action of $\LLx \times \GL(2,\LL)$ on bivariate vectorial Boolean functions of type
\[
f : \LL \times \LL \to \LL 
\]
are important for this paper. The set of $(q,q')$-biprojective functions 
is fixed (setwise) by the action $(\LLx \times \LLx) \times \GL(2,\LL)$ 
(actually, $(q,q')$-biprojectivity is defined in such a way to accommodate this property). 
This and other niceties introduced by $\PGL(2,\LL)$ on $q$-projective functions
allows one to prove the rather straightforward fact \cite[Lemma 3.6]{FG22IEEE} that  
whether these functions are APN can be checked in a simple way using 
parametrization from $\ProjectiveLine(\LL)$ instead of $\LL \times \LL$
which also hints why we have many such families. 
A method to find new biprojective APN families
was given as well in  \cite[Lemma 4.2]{FG22IEEE} along with two biprojective APN families. 
These group actions and biprojectivity were instrumental in giving a method 
(in joint works with Lukas K\"olsch) to check 
equivalences between biprojective APN functions \cite[Theorem 3]{GK21} and 
isotopisms between biprojective semifields \cite[Theorem 5.10]{GK22}, which 
allowed us to solve the equivalence problem for all known biprojective APN 
families and, in the semifields case, generalizing a result of Albert 
\cite{Albert} to give a solution to a 60-year old problem of Hughes 
\cite{Hughes} on determining the autotopism group of Knuth semifields 
\cite{GKPP}. 
The first ever family of commutative semifields of 
odd order that contains an exponential number of non-isotopic semifields 
\cite[Corollary 6.4]{GK22}, and a family containing an exponential number of inequivalent
APN functions \cite[Theorem 5]{GK21} were given recently using this method. 
In the case of APN functions, Kaspers and Zhou 
were the first \cite{KZ} to prove such a result with a different method. 
They showed that the Taniguchi family \cite{Taniguchi}, which is also 
$(q,q')$-biprojective, contains an exponential number of inequivalent APN 
functions. In the $(q,q)$-biprojective 
case, even the larger action $\GL(2,\LL) \times \GL(2,\LL)$ fixes the set of 
$(q,q)$-biprojective functions (setwise). This fact can be seen as the main reason that
the classification in this paper is possible which we will heavily exploit
(see Section \ref{sec_3} ff.) together with the above mentioned properties introduced
by these classical groups on $(q,q)$-biprojective functions.
Note that these are natural and well-known group actions which were used
previously on biprojective semifields \cite{bbeven,bbodd}.

Of course, the study of bivariate functions and semifields (not necessarily 
biprojective) has a long history. In the case of semifields, the bivariate idea 
dates back almost a century to Dickson \cite{Dickson}, Hughes and Kleinfeld 
\cite{HK} and Knuth \cite{Knuth}.
For the APN functions, the initial work on bivariate functions was by Carlet 
\cite{CarletBivariate} who found the first biprojective APN family, and then 
Zhou and Pott \cite{ZP} introduced another family of biprojective APN functions
as well as a family of commutative semifields that contains a quadratic number of 
inequivalent members. Carlet, then introduced \cite{CarletSCM} a method to find 
bivariate (but not necessarily biprojective) APN functions from his previous 
biprojective family. Further work on biprojective APN functions includes the 
family of Taniguchi \cite{Taniguchi}, and quite recently the papers 
\cite{LiZhou,CLV} 
which derive both biprojective and non-biprojective APN families from the 
family of \cite{FG22IEEE} (or by extending it). The method of \cite{CarletSCM} 
was further investigated in \cite{CBC}. Further work that do not involve 
constructions of bivariate functions but study their important 
properties include \cite{Anbar,KZ2,KKK,Kaspers}.

In Section \ref{sec_2} we will explain the notions related to vectorial 
Boolean functions, including the definitions of biprojective functions,
Dembowski-Ostrom polynomials and APN functions that we mentioned in Introduction. 
In Section \ref{sec_3} we will explain the various actions we introduced 
above (and more) on vectorial functions and the corresponding equivalence 
relations including $\approx_{\CCZ}$ and $\GLEQ$ mentioned above. We also 
determine the equivalence classes of one of these actions which will give 
us a \textit{representative set} of biprojective functions which will 
reduce the problem of classification of all biprojective functions to 
the problem of classification of functions in the representative set.
Section \ref{sec_4} contains results required in the proof of the main 
theorem related to zeroes of projective polynomials, Dillon-Dobbertin 
difference sets and the recent classification of fractional projective 
permutations over finite fields. Finally, in Section \ref{sec_5}, we prove 
our main theorem.

\section{Preliminaries}\label{sec_2}

Let $F : \f{2}^n \to \f{2}^n$ be a \textbf{vectorial Boolean function}.
A function $F$ is said to be \textbf{almost perfect nonlinear (APN)}
if 
\[
	F(x) + F(x+a) = b
\]
has zero or two solutions for every
$(a,b) \in \f{2}^n \setminus \{0\} \times \f{2}^n$. Every vectorial
Boolean function $F : \f{2}^n \to \f{2}^n$ can be written as 
an evaluation function of a polynomial over $\f{2^n}[X]$ with 
(polynomial) degree at most $2^n-1$, i.e., $F : X \mapsto F(X)$ where
\[
 F(X) = \sum_{i = 0}^{2^n-1} A_i X^i.
\]
We will not make distinction between functions and their polynomial
representations. The \textbf{algebraic degree} of a vectorial Boolean
function is defined to be
\[
\algdeg(F) = \max \{ \wt_2(i) \ : \ A_i \ne 0 \},
\]
where $\wt_2$ denotes the weight of base-$2$ representation of an integer.
When we say \textit{quadratic}, \textit{affine} or \textit{linear}, we refer to this
notion of degree. We will need the usual \textbf{polynomial degree}
as well which will be denoted by 
\[
\deg(F) = \max \{ i \ : \ A_i \ne 0 \}
\]
as usual. 
The polynomials in $\f{2^n}[X]$ that correspond to \textbf{affine} functions are
\[
M(X) = \sum_{0 \le i \le n-1} C_i X^{2^i} + D.
\]
If $D = 0$, they are called \textbf{linear} functions and are $\f{2}$-linear 
vector-space endomorphisms of $\f{2^n}$ when viewed as a $\f{2}$-vector space.
As polynomials in $\f{2^n}[X]$, they are known as \textbf{linearized polynomials}.
We are particularly interested in the class of quadratic polynomials in $\f{2^n}[X]$,
\[
 Q(X) = \sum_{0 \le i \ne j \le n-1} B_{ij} X^{2^i+2^j} +
       \sum_{0 \le i \le n-1} C_i X^{2^i} +
			 D.
\]
The subclass of the above class of polynomials that contain only the 
non-affine parts is known as \textbf{Dembowski-Ostrom (DO)} polynomials, i.e.,
\[
D(X) = \sum_{0 \le i \ne j \le n-1} B_{ij} X^{2^i+2^j}.
\]
When $n = 2l$, a further subclass of DO polynomials is important. 
The class of polynomials 
\begin{equation}\label{eq_do}
R(X) = A X^{q+1} + B X^{2^l(q+1)} + C X^{q + 2^l} + D X^{2^lq+1},
\end{equation}
where $q = 2^k$ with $0 < k < l$, satisfying the \textbf{subfield property},
\[
	R(aX) = a^{q+1}R(X), \quad \textrm{ for all } a \in \f{2^l},
\]
contains (up to $\CCZ$-equivalence) the only known APN function 
$\kappa$ that is $\CCZ$-equivalent to a
permutation for even $n$ (we will describe various notions of 
equivalences of vectorial Boolean functions further below). Let $\LL = \f{2^l}$.
Identifying $\f{2^{2l}} = \LL(\xi) = \LL\xi + \LL \cong \LL \times \LL$,
we can write the above class of functions 
$
X \mapsto R(X)
$
as
$
	(x,y) \mapsto R(x,y)
$
with
\begin{align*}
	R(x,y) &= ((a_0 x^{q+1} + b_0 x^qy + c_0 xy^q + d_0 y^{q+1}),
	          (a_1 x^{q+1} + b_1 x^qy + c_1 xy^q + d_1 y^{q+1})) \\
				 &= (f(x,y),g(x,y))
\end{align*}
where $X = (x\xi + y)$. This motivates the following definition.

\begin{definition}
Let $q = 2^k$ with $0 < k < l$.
\begin{itemize}
\item
Let $f \in \LL[x,y]$ be a polynomial of the form
\[
	f(x,y) = a_0 x^{q+1} + b_0 x^qy + c_0 xy^q + d_0 y^{q+1}.
\]
Then $f$ is called a \textbf{$q$-biprojective polynomial}.

\item
Let $R : \LL \times \LL \to \LL \times \LL$ be a function of the form
\[
R : (x,y) \mapsto R(x,y) = (f(x,y),g(x,y))
\]
where $f$ and $g$ are $q$-biprojective polynomials. Then $R$ is called a
\textbf{$(q,q)$-biprojective function}.
\end{itemize}
\end{definition}

We will use the shorthand notation 
\[
f =  (a_0,b_0,c_0,d_0)_q,
\]
and
\[
R = ((a_0,b_0,c_0,d_0)_q,(a_1,b_1,c_1,d_1)_q) = (f,g).
\]

The subfield property can be recognized in this form. For a $(q,q)$-biprojective map $F = (f,g)$, one has 
\[
	(f(cx,cy),g(cx,cy)) = c^{q+1}(f(x),g(y)), \quad \textrm{ for all } c,x,y \in \LL. 
\]

\begin{remark}
One can generalize the concept to 
\textbf{$(q,q')$-biprojective functions}
\[
F : (x,y) \mapsto F(x,y) = ((a_0,b_0,c_0,d_0)_q, (a_1,b_1,c_1,d_1)_{q'}),
\]
where $q' = 2^{k'}$ (here one allows $0 \le k,k' < l$). 
These functions satisfy a modified form of the subfield property
\[
	(f(cx,cy),g(cx,cy)) = (c^{q+1}f(x),c^{q'+1}g(y)), \quad \textrm{ for all } c,x,y \in \LL. 
\]
In this paper, we are only interested in $(q,q)$-biprojective functions. See \cite{FG22IEEE,GK21} 
for more on $(q,q')$-biprojective APN functions and their constructions, equivalences and enumerations.

\end{remark}

Define
\begin{align*}
 \DE{f}{0}(x,y) &= b_0 x^q + c_0 x + d_0 y^q + d_0 y,\\
 \DE{f}{\infty}(x,y) &= a_0 x^q + a_0 x + c_0 y^q + b_0 y,\\
 \DE{f}{u}(x,y) &= (a_0 u + b_0) x^q + (a_0 u^q + c_0) x 
	 			         + (c_0 u + d_0) y^q + (b_0 u^q + d_0) y, 
\end{align*}
for $u \in \LLx$.
The following lemma was proved in \cite{FG22IEEE}.

\begin{lemma}\label{lem_cond}
Let $F = (f,g)$ be a $(q,q')$-biprojective function. Then $F$ is APN if and only if 
$
\DE{f}{u}(x,y) = 0 = \DE{g}{u}(x,y)
$
has exactly two solutions for each $u \in \ProjectiveLine(\LL)$.
\end{lemma}

\subsection{The trace map and Hilbert's Theorem 90}
Let $\LL$ be the finite field with $p^l$ elements, $q = p^k$ for $k > 0$
and let $\DD \subset \LL$ of order $p^\delta$ with $\delta = \Gcd(k,l)$. 
The \textbf{trace map} is defined as
\[
\tr{\LL/\DD}{x} = \sum_{j = 0}^{l/\delta-1} x^{(p^\delta)^j}.
\]
When $\DD = \f{p}$ then we simply write
\[
\trace{x} = \tr{\LL/\f{p}}{x}. 
\]

The following is the finite fields version of Hilbert's Theorem 90.

\begin{lemma}[Hilbert's Theorem 90]\label{lem_hilbert}
Let $\Gcd(j,l) = 1$ and $a \in \LL$. Then 
$\tr{\LL/\DD}{a} = 0$ if and only if $a = x^{(p^\delta)^j} - x$ for some $x \in \LL$.
\end{lemma}

The $\DD$-linear vector-space endomorphisms of $\LL$ can be written as
\[
L(x) = \sum_{j = 0}^{l/\delta-1} a_j x^{(p^\delta)^j}, \quad a_j \in \LL,
\]
and are called \textbf{$\DD$-linearized polynomials}. Determining kernels of such 
endomorphisms in $\LL$, especially of the form $L(x) = ax^q - bx$ and the 
zeroes of its translates $L(x)+c$, is important for this paper. This can 
simply be done by observing
\[
ax^q - bx = 0
\]
for some nonzero $r \in \LL$ if and only if $r^{q-1} = b/a$. In that case
$L(rx)/rb = x^q - x$.

Then one can deduce that the zeroes of $L(x)$ are $0$ and $\epsilon r$ for 
$\epsilon \in \DDx$ if such $r$ exists. Then the case $L(x)+c$ can be handled 
using Hilbert's Theorem 90. The following lemma is relevant and will be needed.

\begin{lemma}\label{lem_gcd}
For a prime $p$,
\begin{enumerate}
\item $\Gcd(p^k-1,p^l-1) = p^{\Gcd(k,l)}-1$.
\item 
\[
\Gcd(p^k+1,p^l-1) = \left\{\begin{array}{ll}
1 & \textrm{if } \frac{l}{\Gcd(k,l)} \textrm{ is odd, and } p = 2,\\
2 & \textrm{if } \frac{l}{\Gcd(k,l)} \textrm{ is odd, and } p \textrm{ is odd},\\
p^{\Gcd(k,l)}+1  & \textrm{if } \frac{l}{\Gcd(k,l)} \textrm{ is even}.
\end{array}\right.
\]
\end{enumerate}
\end{lemma}

\subsection{Difference sets with Singer parameters and multisets}

Let $M$ be a multiset whose elements are $s_i$ with repetition $d_i$ for $1 \le i \le r$, denoted by
\[
M = \lms s_1^{[d_1]}, \ldots, s_r^{[d_r]} \rms,
\]
where we write $\mult_M(s_i) = d_i$. 
When all $s_i$ have the same repetition number $d$, we write $M = S^{[d]}$,
where $S = \{ s_1, \ldots, s_r \}$.

For two multisets $S,T$ we denote by
\[
S/T = \lms s/t \ : \ s \in S, t \in T \rms
\]
the \textit{direct division} of two multisets.

A subset $D$ of cardinality $k$ of a group $G$ (written multiplicatively) 
of order $v$ is said to be a \textbf{$(v,k,\lambda)$-difference set} if
\[
D/D = \{1\}^{[k]} \cup (G \setminus \{1\})^{[\lambda]}.
\]

We are mostly interested in the cyclic difference sets in the multiplicative
group $\LLx$ in characteristic $2$. The parameters 
$(2^l-1,2^{l-1},2^{l-2})$ and
$(2^l-1,2^{l-1}-1,2^{l-2}-1)$ are called \textbf{Singer parameters}, since they
are the parameters of the Singer sets (i.e., hyperplanes of $\LL$),
\[
\Hbar = \{ x \in \LL \ : \ \trace{x} = 1 \},
\]
and $\HH \setminus \{0\}$ where
\[
\HH = \{ x \in \LL \ : \ \trace{x} = 0 \},
\]
respectively. Dillon-Dobbertin difference sets are another important example of
cyclic difference sets with Singer parameters \cite{DD} (see Lemmas \ref{lem_dd}, \ref{lem_multiset}) 
which we will use frequently in this paper.

\subsection{Equivalences of vectorial Boolean functions}

Let $F : \FF \to \FF$ be a vectorial Boolean function and $\FF = \f{2^n}$.
The widest known notion of equivalence that keeps the APN property invariant
is called the \textbf{$\CCZ$-equivalence} \cite{CCZ}. Define the \textbf{graph} 
of the function $F$ by
\[
\Gamma_F = \{ (x,F(x)) \ : \ x \in \FF \}.
\]
Then $F$ is said to be $\CCZ$-equivalent to $F'$ if there exists $\f{2}$-linear endomorphisms
$A,B,C,D$ of $\FF$, elements $u,v \in \FF$ and a permutation $\pi : \FF \to \FF$ such that
\begin{equation}\label{eq_CCZ}
\begin{pmatrix}
A & B \\
C & D
\end{pmatrix}
\begin{pmatrix}
x\\
F(x)
\end{pmatrix}
+
\begin{pmatrix}
u\\
v
\end{pmatrix}
=
\begin{pmatrix}
\pi(x)\\
F'(\pi(x))
\end{pmatrix}.
\end{equation}
In that case we write $F \approx_{\CCZ} F'$. A narrower notion of equivalence that keeps
the algebraic degree of $F$ invariant is called the \textbf{extended affine (EA) equivalence}.
We write $F \approx_{\EA} F'$ if 
\[
A_1 \circ F \circ A_2(x) + A_3(x) = F'(x),
\]
for affine maps $A_1,A_2,A_3 : \FF \to \FF$ with $A_1,A_2$ bijective. It can easily be shown 
that $\EA$-equivalence is a special case of $\CCZ$-equivalence where one sets $B = 0$.
An important theorem for quadratic APN functions is that for two quadratic APN functions 
$F,G$ we have by a result of Yoshiara \cite{Yoshiara},
\[
F \approx_{\EA} G \iff F \approx_{\CCZ} G.
\]
We can restrict the equivalence even more if, for instance, we want to keep the property 
of being a permutation invariant. 
The functions $F$ and $F'$ are said to be \textbf{linearly equivalent} if 
\[
L_1 \circ F \circ L_2 = F',
\]
for $L_1,L_2 \in \GL(n,2)$. Note that this is equivalent to setting $B = C = 0$ and 
$u = v = 0$ in \eqref{eq_CCZ} 
(when only $B = C = 0$ holds, they are called \textbf{affinely equivalent}).
We denote this equivalence by $F \approx_{\GL(n,2)} F'$. In the following, we will 
even restrict this equivalence to the case where $L_1,L_2 \in \GL(2,2^l) < \GL(2l,2)$ 
when $n = 2l$ with the obvious motivation that the (left and right) action of 
$\GL(2,2^l)$ keeps the $(q,q)$-biprojective property of $F(x,y) = (f(x,y),g(x,y))$ invariant.


\section{Actions of $\GL(2,\LL)$ and $\PGL(2,\LL)$}\label{sec_3}

In this section we outline the basics of several actions of $\GL(2,\LL)$ and $\PGL(2,\LL)$
on biprojective functions. Most of this section can be considered standard lore. 
Let 
\[
\Vset = \{ (a,b,c,d)_q \ : \ a,b,c,d \in \LL \}, 
\]
be the set of all $q$-biprojective polynomials. 
Let $f, g \in \Vset$ be two $q$-biprojective polynomials and 
\begin{align*}
F : \LL \times \LL &\to \LL \times \LL\\
            (x,y)  &\mapsto (f(x,y),g(x,y))
\end{align*}
be the associated $(q,q)$-biprojective function.
Define $\Fset$ to be the set of all $(q,q)$-biprojective functions, i.e.,
\[
\Fset = \Vset \times \Vset.
\]
Let $\LLinear(\LL)$ be the group of all nonsingular $\LL$-linear transformations of $\LL\times \LL$, i.e.,
\[
\GL(2,\LL) \cong \LLinear(\LL) = \left\{ (x,y) \mapsto (tx+uy,vx+wy) \ : \ t,u,v,w \in \LL \ | \ tw - uv \ne 0 \right\}.
\]
We are mainly interested in the standard action of $\GL(2,\LL) \times \GL(2,\LL)$ on
$(q,q)$-biprojective functions $F \in \Fset$. That is
\[
F'(x,y) = L_1 \circ F \circ L_2(x,y).
\]
This action defines an equivalence relation which we denote by
$F' \GLEQ F.$
Define also the action of the group $\LLx \times \GL(2,\LL)$ on $q$-biprojective polynomials $f \in \Vset$ 
where $\LLx < \GL(2,\LL)$ acts on $f$ by \textbf{scaling}, and the action of $\GL(2,\LL)$ is the usual
right action, i.e.,
\begin{align*}
f'(x,y) &= \alpha ( f \circ L(x,y) ),\\
        &= \alpha \left( a (tx+uy)^{q+1} + b (tx+uy)^q (vx+wy) + c(tx+uy) (vx+wy)^q + d (vx+wy)^{q+1} \right),
\end{align*}
where $(\alpha,L) \in \LLx \times \GL(2,\LL)$. In this case we say that $f \Leq f'$.

Set $\phi_f(x) = f(x,1)$. A polynomial $\phi_f \in \LL[x]$ for $f \in \Vset$ is
called a \textbf{$q$-projective polynomial}.
The projective version of the above action on the bivariate $q$-projective polynomial $f$, 
on the univariate $q$-projective polynomial $\phi_f$ can be given using the 
\textbf{fractional linear (M\"obius) transformations} over the finite field $\LL$, i.e.,
\[
\PGL(2,\LL) \cong \Mobius(\LL) = \left\{ x \mapsto \frac{tx+u}{vx+w} \ : \ t,u,v,w \in \LL \ | \ tw - uv \ne 0 \right\}.
\]
Now, define the action
\begin{equation}\label{eq_action1}
\phi_{f'}(x) = \alpha (vx+w)^{q+1} (\phi_f \circ \mu(x)),
\end{equation}
for $(\alpha,\mu) \in \LLx \times \PGL(2,\LL)$ and 
$
\mu : x \mapsto
\frac{tx+u}{vx+w}.
$ 
One addresses the zero of the denomiator $(vx+w)$ by introducing $\infty = \beta/0$ 
for all $\beta \in \LLx$. We define $\ProjectiveLine(\LL) = \LL \cup \{\infty\}$. 
By defining  $\mu(\infty)=t/v$, we see that all $\mu \in \Mobius(\LL)$ permutes 
$\ProjectiveLine(\LL)$. Note that we view the action as
\begin{equation}\label{eq_action2}
\phi_{f'}(x) = \alpha \left( a (tx+u)^{q+1} + b (tx+u)^q (vx+w) + c(tx+u) (vx+w)^q + d (vx+w)^{q+1} \right),
\end{equation}
so that $\phi_{f'}$ is a ($q$-projective) polynomial over $\LL[x]$ and we do not have to deal with $\infty$ (but we will do that later, since it is helpful). We write $\phi_f \Peq \phi_{f'}$. The following lemma is straightforward.

\begin{lemma}\label{lem_unibi}
We have 
$
	f \Leq f' \textrm{ if and only if } \phi_f \Peq \phi_{f'}.
$
\end{lemma}
We will not use $\phi_f$ to refer to univariate $q$-projective version of $f$ and
instead use $f$ for both univariate and bivariate functions and polynomials. 
We will also use $\Vset$ as the ambient space of both types of functions/polynomials.

Our aim is to classify the APN functions in $\Fset$ under the equivalence $\GLEQ$. 
We say that $\Sset \subseteq \Vset$ is a {\bf representative set} of $\Vset$
if (denoting by $[f]_\sim$ the equivalence class of $f$ under $\sim$)
\[
[\Sset]_\Peq = \bigcup_{f \in \Sset} [f]_\Peq = \Vset.
\]

In the following we will find a representative set $\Sset$ using univariate 
$q$-projective polynomials and their equivalence $\Peq$. Then it will be clear that checking
\[
(f,g) \in \Sset \times \Vset
\]
is enough for our classification under $\GLEQ$ since $f \Peq f'$ if and only if $f \Leq f'$
and if $F = (f,g) \in \Vset \times \Vset$ then 
$F \GLEQ F' = (f',g') \in \Sset \times \Vset$ for some $f',g'$ using the 
$\LL$-linear transformation implied by the equivalence $f \Leq f'$.

\subsection{Roots of projective polynomials}

Let $q = p^k$, $\DD = \f{p^\delta}$, $\LL = \f{p^l}$ where $\delta = \Gcd(k,l)$ and
\[
f(x) = ax^{q+1} + bx^q + cx + d
\]
be a nonzero $q$-projective polynomial. 
If $a = 0$, then $f$ is an affine polynomial and the set of zeroes of $f$ in $\LL$, i.e.,
\[
Z'_f = \{ x \in \LL \ : \ f(x) = 0\}
\]
satisfies $|Z'_f| \in \{ 0, 1, p^\delta$\}. To see that, first 
observe that 
\textbf{scaling}, i.e., $f \mapsto \alpha f$ for $\alpha \in \LLx$,
the \textbf{translations} $f(x) \mapsto f(x + \beta)$ for $\beta \in \LL$ and 
the \textbf{dilations}    $f(x) \mapsto f(\gamma x)$ for $\gamma \in \LLx$ 
keep the number of zeroes of $f$ in $\LL$ invariant. Then we have only 
a few options to consider:
\begin{itemize}
\item $f = 1$ 

has no $\LL$-zeroes ---degenerate case (together with omitted $f = 0$).
\item $f \in \{x,x^q\}$

has one $\LL$-zero.
\item $f = x^q-cx-d$ where $c \ne 0$. 

We have 
\begin{itemize}
\item if $c = A^{q-1} \in (\LLx)^{q-1}$, then $f$ has 
\begin{itemize}
\item $p^\delta$ $\LL$-zeroes if $\tr{\LL/\DD}{d/A^q} = 0$, and 
\item no $\LL$-zeroes if $\tr{\LL/\DD}{d/A^q} \ne 0$; or
\end{itemize}
\item one $\LL$-zero if $c \not\in (\LLx)^{q-1}$,
\end{itemize}
by Hilbert's Theorem 90.
\end{itemize}
Now assume $a \ne 0$. We will show that
$|Z'_f| \in \{ 0, 1, 2, p^\delta + 1\}.$
Assume $f$ has at least one $\LL$-zero $r \in Z'_f$. Now consider
\[
f'(x) = f(x+r) = ax^{q+1}+b'x^q+c'x.
\]
The \textbf{reciprocal} $f''(x) = x^{q+1}f'(1/x)$ is
\[
f''(x) = a+b'x+c'x^q,
\]
where $\deg f'' < q+1$. Now $f''$ has one fewer $\LL$-zeroes than $f'$ 
(since we \textit{punctured} the zero of $f'$ at $0$) and thus we have
\[
|Z'_{f''}|+1 = |Z'_{f'}| = |Z'_f| \in \{ 0, 1, 2, p^\delta + 1\}.
\]

This observation, together with the fact that $\PGL(2,\LL)$ is generated by translations, dilations 
and inversions (for which the corresponding action is reciprocation), 
motivates us to ascribe $f(\infty) = 0$ if and only if $\deg f < q+1$ so that $\Peq$ 
preserves the number of roots in $\ProjectiveLine(\LL)$.
Now we can define 
\begin{definition}
The \textbf{$\PLL$-zeroes} of a $q$-projective polynomial $f$ is defined as
\[
Z_f = \{ x \in \ProjectiveLine(\LL) \ : \ f(x) = 0\},
\]
where we define $f(\infty) = 0$ if and only if $\deg f < q+1$.
\end{definition}

Thus we have
\begin{lemma}\label{lem_zeroes}
Let $f \ne 0$ be a $q$-projective polynomial. Then, 
\begin{enumerate}
\item $|Z_f|$ is invariant under $\Peq$, and 
\item $|Z_f| \in \{ 0, 1, 2, p^\delta + 1\}$.
\end{enumerate}
\end{lemma}
\begin{proof}
Part (ii) is explained before the lemma. We will give a more detailed 
proof of Part (i) which, actually, also follows from the previous discussion.

Let $f' \Peq f$ via $(\alpha,\mu) \in \LLx \times \PGL(2,\LL)$ with
$\mu : x \mapsto \frac{tx+u}{vx+w}.$ Suppose $v \ne 0$.
Define $T = \LL \setminus \{ w/v \}$,
which satisfies $\mu(T) = \LL \setminus \{t/v\}$. 
We see by \eqref{eq_action1} that 
\[
x \in T \textrm{ is a zero of $f'$ } \iff \mu(x) \textrm{ is a zero of } f.
\]
By \eqref{eq_action2}, we get $a = 0$ if and only if
\[
w/v = \mu^{-1}(\infty) \textrm{ is a zero of $f'$ } \iff  \infty \textrm{ is a zero of } f
\]
if and only if $\deg f < q+1$. Also by \eqref{eq_action2}, since the coefficient of 
the $x^{q+1}$ term of $f'$ is $a t^{q+1} + b t^qv + ctv^q + d v^{q+1} = v^{q+1}f(t/v)$,
we have
\[
t/v = \mu(\infty)      \textrm{ is a zero of $f$ } \iff \infty \textrm{ is a zero of } f'
\]
if and only if $\deg f' < q+1$.  Combining the three cases we get $\mu(Z_{f'}) = Z_f$ and $|Z_f| = |Z_{f'}|$.
These facts also trivially follow when $v = 0$.
\end{proof}

Define the sets (suppressing $q$ and $\LL$ from the notation for simplicity),
\begin{align*}
D_0 &= \left\{ (0,0,0,0)_q \right\}, \\
D_1 &= \left[ (0,0,0,1)_q \right]_\Peq, \\
D   &= D_0 \cup D_1,\\
\Pi_j &= \left\{ f \in \Vset \setminus D \ : \ |Z_f| = j \right\}, 
\end{align*}
for $j \in \{ 0, 1, 2, p^\delta + 1\}$. 

\begin{lemma}\label{lem_13}
We have
\begin{enumerate}
\item $\Vset = \bigcup_{j \in \{0, 1, 2, p^\delta + 1\}} \Pi_j \cup D$ ,
\item $\Pi_1 = [ (0,1,-1,-u)_q ]_\Peq$ where $u \in \LL$ satisfies $\tr{\LL/\DD}{u}=1$,
\item $\Pi_{{p^\delta+1}} = [ (0,1,-1,0)_q ]_\Peq$, 
\end{enumerate}
\end{lemma}
\begin{proof}
The proof of Part (i) follows from the discussion before the lemma. For Parts (ii) and (iii), we have to 
show that the action is transitive on the sets $\Pi_1$ and $\Pi_{p^\delta+1}$. If $a \ne 0$, as in the
discussion before the lemma, using $f'(x) = f(x+r)$ first, where $r$ is an $\LL$-zero of $f$, 
and then using the reciprocal $f''(x) = x^{q+1}f'(1/x)$, one shows that $f \Peq f''$ with 
$\deg f'' < q+1$. Now in both $p^\delta+1$ and one $\PLL$-root cases, $f''$ has the form
\[
f''(x) = a+b'x+c'x^q,
\]
with $-b'/c' = A^{q-1}$ for some $A \in \LLx$. One can apply $x \mapsto xA$ and then scale 
to get $f'' \Peq x^q - x - d$ for some $d \in \LL$,
where $f$ has $p^\delta+1$ $\PLL$-roots if and only if $\tr{\LL/\DD}{d} = 0$ and 
$f$ has one $\PLL$-root if and only if $\tr{\LL/\DD}{d} \ne 0$ by Hilbert's Theorem 90, and
the action is transitive on $\Pi_{p^\delta+1}$ since any $d$ with $\tr{\LL/\DD}{d} = 0$ can be written
as $z^q - z = d$ for some $z \in \LL$ and $f \Peq x^q - x$ after the application of $x \mapsto x+z$. 

For the case $\tr{\LL/\DD}{d} \ne 0$,
applying $x \mapsto \epsilon x$ where $\epsilon \in \DDx$, and scaling by $1/\epsilon$, one can 
assume $\tr{\LL/\DD}{d} = 1$. Then the transitivity of the action follows again by translations.
\end{proof}

When $\Gcd(p^k-1,p^l-1) = 1$, that is to say $\Gcd(k.l)=1$ and $p = 2$ by 
Lemma \ref{lem_gcd}, we can say more.

\begin{lemma}\label{lem_0123}
Let $p = 2, q = 2^k$, and $\delta = \Gcd(k,l) = 1$, then we have
\begin{enumerate}
\item $\Pi_0 = [ g ]_\Peq$, where $g \in \Pi_0$,
\item $\Pi_2 = [ (0,0,1,0)_q ]_\Peq$. 
\item If $l$ is odd, then
\begin{enumerate}
\item $\Pi_1 = [ (1,0,0,1)_q ]_\Peq$.
\end{enumerate}
\item If $l$ is even, then
\begin{enumerate}
\item $\Pi_0 = [ (1,0,0,u)_q ]_\Peq$, where $u \in \LLx \setminus (\LLx)^{q+1}$,
\item $\Pi_3 = [ (1,0,0,1)_q ]_\Peq$.
\end{enumerate}
\end{enumerate}
\end{lemma}
\begin{proof}
\begin{enumerate}
\item This was proved in
\begin{itemize}
\item \cite[Theorem 2.1]{BTT} when $l$ is even, and
\item \cite[Lemma 7]{GK21} for general $l$.
\end{itemize}
\item Since $\Gcd(2^k-1,2^l-1) = 2^{\Gcd(k,m)}-1 = 1$ by Lemma \ref{lem_gcd}, the only
option (under translations and dilations) when $a=0$ is $f \in \{x,x^q\}$ by the discussion preceding the lemma. 
Note that $x^{q+1}(1/x^q) = x$, hence they are reciprocals of each other. 
If $a = 1$, then using a translation and reciprocation, one gets $f \Peq f''$ where
$\deg f'' < q+1$. Therefore the action is transitive.
\item This is basically Lemma \ref{lem_13} (ii) and the fact that 
$x^{q+1} + u$ has one root for all $u \in \LLx$ by Lemma \ref{lem_gcd}.
\item These follow from Lemma \ref{lem_13} (iii) and Part (i) of this lemma
together with the fact that $x^{q+1}+u$ has either three or no roots
for the given conditions on $u \in \LLx$ by Lemma \ref{lem_gcd}.
\end{enumerate}
\end{proof}

We can now state the representative set $\Sset$ we will use.
\begin{lemma}\label{lem_sset}
Let $p = 2, q = 2^k$, $\delta = \Gcd(k,l) = 1$, and
\[
S = \{ (0,0,0,0)_q, (0,0,0,1)_q, (0,0,1,0)_q \} \cup  
    \{ (1,0,0,a)_q \ : \  a \in \LLx \}.
\]
\begin{enumerate}
\item If $l$ is odd then
\[
\Sset = S \cup \{ (0,1,1,0)_q \} \cup \Pi_0.
\]
\item If $l$ is even then
\[
\Sset = S \cup \Pi_1.
\]
\end{enumerate}
Then $\Sset$ is a representative set for $\Vset$.
\end{lemma}


\section{Further lemmas}\label{sec_4}

In this section we will give the results that are needed in the classification.

\subsection{The zeroes of $x^{q+1}+x+b$}

Now we are going to restrict ourselves to a specific type of $q$-projective polynomials, namely
\[
P_b(x) = x^{q+1} + x + b,
\]
for $b \in \LL$. Bluher studied these polynomials \cite{Bluher} and determined 
the cardinalities of the sets 
\begin{align*}
I_j = \{ b \in \LL \ : \ P_b \in \Pi_j \},
\end{align*}
where $j \in \{0,1,2,p^\delta+1\}$. In this section we are interested only in the $p=2$ 
case and $\Gcd(k,l)=1$ where $\LL = \f{2^l}$ and $q = 2^k$. In the definitions of $I_j$
and $P_b$ (along with previously defined sets) we suppress $q$ and $\LL$ for notational simplicity.

The following is an important result on combinatorics of finite fields proved by 
Dillon and Dobbertin \cite{DD}.

\begin{lemma}\label{lem_dd}
Let $q = 2^k$ and $\Gcd(k,l) = 1$.
The set $I_1$ is a
\begin{enumerate} 
\item $(2^l-1,2^{l-1}-1,2^{l-2}-1)$-difference set in $\LLx$ if $l$ is odd, and
\item $(2^l-1,2^{l-1},2^{l-2})$-difference set in $\LLx$ if $l$ is even.
\end{enumerate}
\end{lemma}
\begin{proof}
Let $d = 4^k - 2^k +1$ and
\[
\Delta = \left\{
\begin{array}{ll}
               \{ x^d + (x+1)^d + 1 : x \in \LL \setminus \f{2} \} & \textrm{if $l$ is odd, and}\\
\LLx \setminus \{ x^d + (x+1)^d + 1 : x \in \LL \setminus \f{2} \} & \textrm{if $l$ is even}.
\end{array}
\right.
\]
Dillon and Dobbertin showed \cite[Theorem A]{DD} that $\Delta$ is a difference set with indicated
(Singer) parameters.

The fact that $I_1 = 1 / \Delta$ was shown in \cite[Theorem 1]{HK1} (see also \cite[Theorem 5.13]{FG22FFA}). 
\end{proof}

In the next lemma we show that, since $\Gcd(2^k-1.2^l-1) = 2^{\Gcd(k,l)}-1=1$,
we can easily determine the sets $I_1,I_2$ and $I_3$. The map $\rho$ appears quite 
frequently when one works with projective polynomials, for instance, in Serre's
proof that $\PSL(2,q)$ is the Galois group of the equation $x^{q+1} - xy + 1 = 0$
for arbitrary prime power $q$ (see \cite[pp. 131--132]{Abhyankar92}). 
The lemma can be found for general $\Gcd(k,l)$ and $q = 2^k$ in \cite[pp. 175--176]{HK1}. 
We provide a simple proof for $\Gcd(k,l) = 1$ again in characteristic $2$.

\begin{lemma}\label{lem_imageset}
Let $q = 2^k$, $\Gcd(k,l) = 1$ and
\begin{align*}
	\rho : \LL \setminus \f{2} &\to \LL \setminus \f{2},\\
					                x  &\mapsto \frac{x^{q^2+1}}{(x^q+x)^{q+1}}.
\end{align*}
We have
\begin{enumerate}
\item $\lms \rho(x) \ : \ x \in \Hbar \setminus \f{2} \rms = I_1$,
\item $\lms \rho(x) \ : \ x \in \HH   \setminus \f{2} \rms = I_3^{[3]}$,
\item $I_2 = \{0\}$,
\item $\lms x^{q+1} + x \ : \ x \in \LL \rms = I_1 \cup I_2^{[2]} \cup I_3^{[3]}$.
\end{enumerate}
\end{lemma}
\begin{proof}
Let $b \in I_1 \cup I_2 \cup I_3$ and $f = P_b = x^{q+1}+x+b$. Let $r \in \LL$ be
a root of $f$. Then
\begin{align*}
f(x+r) &= (x+r)^{q+1} + x + r + b\\
       &= x^{q+1} + rx^q + (r+1)^qx + r^{q+1}+r+b \\
       &= x^{q+1} + rx^q + (r+1)^qx\\
	   &\Peq (r+1)^qx^q + rx + 1 = f''(x)
\end{align*}
Then $|Z_f| = 2$ if and only if $|Z_{f''}|=1$ if and only if $r \in \f{2}$. Thus $b = 0$ and Part 
(iii) follows. Otherwise we have $A \in \LLx$ such that 
\[
A^{q-1} = \frac{r}{(r+1)^q},
\]
since $x \mapsto x^{q-1}$ is bijective on $\LL$ by Lemma \ref{lem_gcd}. We have 
$f \in \Pi_1$ if and only if $\trace{1/Ar} = 1$ and 
$f \in \Pi_3$ if and only if $\trace{1/Ar} = 0$ by Hilbert's Theorem 90.
Now, 
let $h \in \LLx$ satisfy
\[
\frac{1}{Ar} = h^q,
\]
That is to say
\[
\frac{1}{A} = h^qr \iff \frac{1}{A^{q-1}} = h^{q(q-1)}r^{q-1} = \frac{(r+1)^q}{r}.
\]
Thus, recalling $r \not\in \f{2}$,
\[
h^{q-1} = \frac{r+1}{r} = 1 + \frac{1}{r}.
\]
Therefore for each $r \in \LL \setminus \f{2}$, there is unique $h \in \LL \setminus \f{2}$,
where $h \in \HH \setminus \f{2}$   if and only if $f \in \Pi_3$
and   $h \in \Hbar \setminus \f{2}$ if and only if $f \in \Pi_1$.
Or, equivalently
\[
r = \frac{h}{h+h^q}.
\]
The remaining  Parts (i), (ii), (iv) follow after observing
\[
b = r^{q+1} + r = \left(\frac{h}{h+h^q}\right)^{q+1} + \frac{h}{h+h^q} 
= \frac{h^{q+1}+h(h^q+h)^q}{(h^q+h)^{q+1}} = \frac{h^{q^2+1}}{(h^q+h)^{q+1}}.
\]
\end{proof}

The following lemma is key to our classification. It will prove 
that the Kim function $\kappa$ exists as a \textit{theorem of small cases},
using the properties of Dillon-Dobbertin difference sets with Singer parameters.

\begin{lemma}\label{lem_multiset}
Let $l > 3$, $q = 2^k$ and $\Gcd(k,l)=1$. For all 
\begin{itemize}
\item $d \in \LLx$ if $l$ is odd, and
\item $d \in \LLx \setminus (\LLx)^{q+1}$ if $l$ is even, we have
\end{itemize}
\[
	d I_3 \cap (I_1 \cup I_2 \cup I_3) \ne \emptyset.
\]
\end{lemma}
\begin{proof}
By Lemma \ref{lem_imageset} (iii), $I_2 = \{0\}$ and therefore $dI_3 \cap I_2 = \emptyset$.
Also the $d = 1$ case is clear. 
Let
\[
M = \frac{I_1 \cup I_3^{[3]}}{I_3^{[3]}}.
\]
The claim of the lemma is equivalent to the claim that for all $d \in \LL \setminus \f{2}$,
\[
\mult_{M}(d) > 0.
\]
We have
\[
M = \frac{I_1}{I_3^{[3]}} \cup \frac{I_3^{[3]}}{I_3^{[3]}},
\]
and
\[
J = \frac{I_1 \cup I_3^{[3]}}{I_1 \cup I_3^{[3]}} = M \cup \frac{I_1}{I_1} \cup \frac{I_3^{[3]}}{I_1}.
\]
It is clear by Lemma \ref{lem_imageset} (iii) and (iv) that
\[
I_1 \cup I_3^{[3]} = \lms x^{q+1} + x \ : \ x \in \LL \setminus \f{2} \rms.
\]
For $x,y,d \in \LL \setminus \f{2}$, we will find the number of solutions of
\[
	\frac{x^{q+1} + x}{y^{q+1} + y} = d,
\]
which is (for $y,d \in \LL \setminus \f{2}$ and $x \in \LLx \setminus \{1/y\}$) 
the same as the number of solutions of
\[
	\frac{(xy)^{q+1} + xy}{y^{q+1} + y} = \frac{x^{q+1}y^q + x}{y^{q} + 1} = d.
\]
Or, equivalently
\[
	y^q(x^{q+1}+d) = (x+d).
\]
For all $x \in \LL \setminus \f{2}$ such that $x^{q+1}\ne d$ or $x \ne d$,
there exists a (unique) $y \in  \LL \setminus \f{2}$. The equality holds when 
$x = 1/y$ if and only if $y^q d = d$, that is to say $x = y = 1$.

Thus
\[
\mult_{J}(d) = \left\{
\begin{array}{ll}
2^l - 4 & \textrm{if $l$ is odd}, \\
2^l - 6 & \textrm{if $l$ is even and $d \in (\LLx)^{q+1}$},\\
2^l - 3 & \textrm{if $l$ is even and $d \in \LLx \setminus (\LLx)^{q+1}$},
\end{array}
\right.
\]
since $\Gcd(2^k+1,2^l-1) = 3$ if $l$ is even and $1$ if $l$ is odd.
Since $I_1$ is a difference set with Singer parameters, we have by Lemma \ref{lem_dd}
\[
\mult_{I_1/I_1}(d) = \left\{
\begin{array}{ll}
2^{l-2}-1 & \textrm{if $l$ is odd}, \\
2^{l-2}   & \textrm{if $l$ is even}.
\end{array}
\right.
\]
Since for all $i \in I_3$ we have at most one $j \in I_1$ since $I_1$ is a set
by Lemma \ref{lem_imageset} (i), we have the trivial bound $\mult_{I_3^{[3]}/I_1}(d) \le |I_3^{[3]}|$.
By Lemma \ref{lem_imageset} (ii), we have,
\[
\mult_{I_3^{[3]}/I_1}(d) \le \left\{
\begin{array}{ll}
2^{l-1}-1 & \textrm{if $l$ is odd}, \\
2^{l-1}-2 & \textrm{if $l$ is even}.
\end{array}
\right.
\]

Note that by the definition of $J$, we must have,
\[
\mult_J(d) = \mult_M(d) + \mult_{I_1/I_1}(d) + \mult_{I_3^{[3]}/I_1}(d).
\]
Now assume $\mult_{M}(d) = 0$. If $l$ is odd, this means
\begin{align*}
2^l - 4 &\le 0 + 2^{l-2}-1 + 2^{l-1}-1,\\
2^{l-2} &\le 2,
\end{align*}
which means $l \le 3$. Similarly for $l$ is even and $d \in \LLx \setminus (\LLx)^{q+1}$,
\begin{align*}
2^l - 3 &\le 0 + 2^{l-2} + 2^{l-1}-2,\\
2^{l-2} &\le 1,
\end{align*}
which means $l \le 2$. 

The claim also holds for $d \in (\LLx)^{q+1}$ similarly but skipped 
since it will not be used and requires computerized check for $l = 4$.
\end{proof}

\subsection{Results on fractional projective permutations of $\PLL$}

Given two $q$-projective polynomials $f(x,1),g(x,1)$ over $\LL$, 
we can define a fractional projective map $x \mapsto f(x,1)/g(x,1)$
on $\ProjectiveLine(\LL)$ whenever $f$ and $g$ do not have a common zero.
Fractional projective permutations over a finite field $\LL$ 
of order $p^l$ has been classified for every parameter $p,k,l$ \cite{FG22FFA}. 
Recall the specific Dembowski-Ostrom polynomials of type \eqref{eq_do}.
The monomial Gold maps $X \mapsto X^s$ where 
$s \in \{ q+1,(q+1)r,qr+1,q+r \}$ on $\f{p^{2l}}$ 
with $q = p^k$ and $r = p^l$ have been shown to be connected to
the fractional projective permutations. Identifying
$\f{p^{2l}} = \LL(\xi) = \LL\xi + \LL \cong \LL \times \LL$, the
above Gold maps can be written as $(q,q)$-biprojective
polynomials $G_s(x,y) = (f_s(x,y),g_s(x,y))$ using $X = x\xi+y$.

The following lemma has been proved for general parameters in \cite{FG22FFA}.
Here we include only the results necessary for our treatment.

\begin{lemma}\label{lem_fracproj}
Let $p = 2$, $q = p^k$, $r = p^l$ and $\Gcd(k,l) = 1$. 
\begin{itemize}
\item If $l$ is odd, then
\[
	x \mapsto \frac{x^{q+1}+c}{x^q+x+d}
\]
permutes $\ProjectiveLine(\LL)$ if and only if $c \in \f{2}$ and $d = 1$.
\item If $l$ is even and $f(x,1) = 0 = g(x,1)$ does not hold for $x \in \PLL$, 
then the fractional projective map
\[
x \mapsto f(x,1)/g(x,1) 
\]
permutes $\ProjectiveLine(\LL)$ if and only if $(f,g) \GLEQ G_{q+1}$ or $(f,g) \GLEQ G_{q+r}$.
\end{itemize}
\end{lemma}

Note that in the even case $G_{q+1}$ and $G_{q+r}$ are APN, since
$\Gcd(k,l) = \Gcd(k,2l) = \Gcd(k+l,2l) = 1$ by \cite{gold} which
states that $G_{2^i+1}$ is APN over $\f{2^n}$ if and only if $\Gcd(i,n)=1$.

\section{The classification}\label{sec_5}

First we will prove a necessary $\Gcd$-condition on a $(q,q)$-biprojective
APN function $F$.

\begin{proposition}\label{prop_gcd}
Let $q = 2^k$ and $\Gcd(k,l)>1$. Then $F \in \Fset$ is not APN.   
\end{proposition}
\begin{proof}
Let $F = (f,g) = ((a_0,b_0,c_0,d_0)_q,(a_1,b_1,c_1,d_1)_q)$.
By Lemma \ref{lem_cond}, if $F$ is APN, then 
\begin{align*}
b_1 x^q + c_1 x &= d_1(y^q + y),\\
b_0 x^q + c_0 x &= d_0(y^q + y),
\end{align*}
has two solutions in $\LL \times \LL$. Solutions to this equation pair include 
$(x,y) \in \{0\} \times \Ker(y^q+y)$. Thus $\Gcd(k,l)=1$. 
\end{proof}

Now we can concentrate on the case $\Gcd(k,l)=1$. We can
assume $(f,g) \in \Sset \times \Vset$ where $\Sset$ is found
in Lemma \ref{lem_sset}. We will first deal with the case $f \in S$
(see Lemma \ref{lem_sset} for the definitions of $S$ and $\Sset$ we use).

\subsection{The case $f \in S$}

\begin{proposition}\label{prop_s}
Let $q = 2^k$, $l > 3$, $\Gcd(k,l) = 1$ and 
\begin{align*}
F : \LL \times \LL &\to \LL \times \LL,\\
            (x,y)  &\mapsto (f(x,y),g(x,y)),
\end{align*}
where $f \in S \subset \Sset$ and $g \in \Vset$. Then $F$ is not APN.
\end{proposition}

\begin{proof}
We will analyze each function $f \in S$ case by case. 
We let $g = (a_1,b_1,c_1,d_1)_q$. 
\begin{itemize}
\item $f \in \{ (0,0,0,0)_q, (0,0,0,1)_q \}$.

We have
$
 \DE{f}{\infty}(x,y) = 0
$
for all $(x,y)$ and $\DE{g}{\infty}(x,y) = 0$ if and only if
\[
a_1 x^q + a_1 x = c_1 y^q + b_1 y.
\]
If $a_1 = 0$ or $b_1 = c_1 = 0$, then the claim easily follows
since $(x,y) \in \LL \times \{0\}$ (or $\{0\} \times \LL$, resp.) satisfy the equality.
Otherwise $|\Image{a_1 x^q + a_1 x} \cap \Image{c_1 y^q + b_1 y}| \ge 2^{l-2}$,
since both image sets are at least $l-1$ dimensional $\f{2}$-vector spaces
which have to intersect at a vector space with dimension at least $l-2$.

\item $f = (0,0,1,0)_q$.

Since $(f,g) \GLEQ (f,rf+sg)$ if and only if $s\ne0$ we can assume that $c_1 = 0$.
The equalities for $\DE{f}{u}(x,y) = 0 = \DE{g}{u}(x,y) = 0$ for $u \in \{0,\infty\}$ give
\begin{align*}
 \DE{f}{0}(x,y) &= x = 0,\\
 \DE{g}{0}(x,y) &= b_1 x^q + d_1 y^q + d_1 y = 0,
\end{align*}
and
\begin{align*}
 \DE{f}{\infty}(x,y) &= y^q = 0,\\
 \DE{g}{\infty}(x,y) &= a_1 x^q + a_1 x + b_1 y.
\end{align*}
imply that $a_1 \ne 0$ and $d_1 \ne 0$. Thus, setting $a_1 = 1$ by the scaling action, 
we will check the common solutions of
\begin{align*}
 \DE{f}{u}(x,y) &=  x + u y^q = 0,\\
 \DE{g}{u}(x,y) &= (u + b_1) x^q + u^q x + d_1 y^q + (b_1 u^q + d_1) y = 0,
\end{align*} 
for $u \in \LL^\times$. Replacing $x$ by $uy^q$ in the second equation
\begin{align*}
 \DE{g}{u}(x,y) &= (u + b_1) (uy^q)^q + u^q (uy^q) + d_1 y^q + (b_1 u^q + d_1) y = 0,\\
&= (u^{q+1}  + b_1 u^q ) y^{q^2} + (u^{q+1}+d_1) y^q + (b_1 u^q + d_1) y = 0,\\
&= u^{q+1}(y^{q^2} + y^q) + b_1 u^q (y^{q^2} + y) + d_1 (y^q + y) = 0.
\end{align*}
Thus for every $u \in \LL^\times$, all $(uy^q,y)$ for $y \in \f{2}$ give a solution. 
To be APN, these should be the only solutions. 
We will use the fact that $y^q + y = h \in \HH \setminus \{0\}$, 
for $y \in \LL \setminus \f{2}$ by Hilbert's Theorem 90.
If $b_1 = 0$, then
\[
\frac{d_1}{u^{q+1}} = \frac{h^q}{h}
\]
has a solution for every $h \in \HH \setminus \{0\}$.
Now, applying $u \mapsto ub_1$ we get, by setting $d^q = d_1/b_1^{q+1}$, 
\begin{align*}
u^{q+1}h^q + u^q (h^q + h) + d^q h       &= 0,\\
hv^{q+1} + v (h^q+h) + d h^q &= 0,
\end{align*}
setting $v = d/u$ and then multiplying the equality by $v^{q+1}/d^q$. 
Now if $1 \in \HH$ (i.e., $l$ is even) then $d \not \in (\LLx)^{q+1}$
for $F$ to be APN. For $h \in \HH \setminus \f{2}$, setting 
$v = x((h^{q}+h)/h)^\qbar$, we get
\[
x^{q+1} + x + \frac{d h^{q+\qbar}}{(h^q+h)^{\qbar+1}} = 0.
\]
We have
\[
\lms h^{q+\qbar}/(h^q+h)^{\qbar+1} \ : \ h \in \HH \setminus \f{2} \rms  = 
\lms h^{q^2+1} /(h^q+h)^{q+1}      \ : \ h \in \HH \setminus \f{2} \rms  =
I_3^{[3]},
\]
by Lemma \ref{lem_imageset} (ii), since $h \in \HH$ if and only if $h^q \in \HH$.
Now $F$ is APN if and only if 
$
  d I_3 \cap (I_1 \cup I_2 \cup I_3) = \emptyset.
$
By Lemma \ref{lem_multiset} we have 
$
	d I_3 \cap (I_1 \cup I_2 \cup I_3) \ne \emptyset,
$
for $l > 3$ and we are done.

\item $f = (1,0,0,d_0)_q$ for $d_0 \in \LLx$.

We can assume $a_1 = 0$. We have
\begin{align*}
 \DE{f}{0}(x,y) &= d_0(y^q + y) = 0,\\
 \DE{g}{0}(x,y) &= b_1 x^q + c_1 x + d_1 (y^q + y) = 0.
\end{align*}
Thus either $c_1 = 0$ or $b_1 = 0$ (but not both), and since
$(0,b_1,0,d_1)_q \Peq (0,0,c_1,d_1)_q \Peq (0,0,1,0)_q$ we are done as
we have already handled that case before.
\end{itemize}
\end{proof}

\begin{remark}
In the proof of the case $f = (0,0,1,0)_q$ we assumed $l > 3$ 
and we used Lemma \ref{lem_multiset} to show that
$F$ is not APN. For $l = 3$, the necessary and sufficient condition
\[
d I_3 \cap (I_1 \cup I_2 \cup I_3) = \emptyset,
\]
holds for $d \in J = \{ \omega, \omega^2, \omega^4 \}$ where
$\omega \in \f{2^3}$ satisfying $\omega^3 + \omega + 1 = 0$.
In this case
\begin{align*}
I_0 &= \{ \omega, \omega^2, \omega^4 \},\\
I_1 &= \{ \omega^3, \omega^5, \omega^6 \},\\
I_2 &= \{ 0 \}, \textrm{ and}\\
I_3 &= \{1\}.
\end{align*}
It is easy to see the direct product satisfies $J I_3 = I_0$.
Thus $F : \f{2^3} \times \f{2^3} \to \f{2^3} \times \f{2^3}$,
with $F = ((0,0,1,0)_2,(1,b_1,0,d_1)_2)$ for $b_1,d_1 \in \f{2^3}$
is APN if and only  if $\omega^{2^i} = d_1/b_1^3$. These are 
precisely the Kim $\kappa$ functions (up to $\Peq$ equivalence).
\end{remark}

\begin{remark}\label{rem_butterfly}
Note that the case $(f,g) \in \{ (1,0,0,1)_q \} \times \Pi_1$ proved as a subcase 
in Proposition \ref{prop_s}
generalizes the results in \cite{CPT,CDP} that show that the class of functions satisfying
the \textbf{(generalized) butterfly structure} are not APN when $l > 3$ is odd. These functions are
by definition \cite{CPT},
\[
	\mathcal{B}_{q,\LL} = \{ ((x+ay)^{q+1}+(by)^{q+1},(y+ax)^{q+1}+(bx)^{q+1}) \ : \ a,b \in \LLx \}.
\]
It is easy to see via $x \mapsto x + ay$ and then $y \mapsto y/b$ for the left component 
(and similarly $y \mapsto y + ax$ and $x \mapsto x/b$ for the right component) 
that $(f,g) \in \mathcal{B}_{q,\LL} \subset \Pi_1 \times \Pi_1$. 

Note that $\Pi_1 \times \Pi_1$ contains functions ($\CCZ$-equivalent to permutations)
that are not contained directly in $\mathcal{B}_{q,\LL}$ and whether the 
$\CCZ$-equivalence class (or $\GLEQ$ class) of $\mathcal{B}_{q,\LL}$ covers all
such functions is not covered in \cite{CPT,CDP} and seems to be difficult to solve.
The question is easy to answer for the right action of $\GL(2,\LL)$ together with
scaling on both components 
(the natural subgroup action $(\LLx \times \LLx) \times \GL(2,\LL)$ that 
preserves inclusion in $\Pi_1$). 
After the transformations ($x \mapsto x + ay$ and then $y \mapsto y/b$) 
on the left part and applying all $\GL(2,\LL)$ transformations stabilizing the 
left part $(1,0,0,1)_q$ (generated by $(x,y) \mapsto (y,x)$) and scaling on the 
right side, we see
\[
\mathcal{B'}_{q,\LL} = \{ (1,0,0,1)_q \} \times 
                \left( \{ h_{a,b,c}(x,y) \ : \ a,b,c \in \LLx \} 
									\cup \{ h_{a,b,c}(y,x) \ : \ a,b,c \in \LLx \} \right),
\]
where
\[
h_{a,b,c}(x,y) = c ((a+1)y/b + ax)^{q+1} + (bx + ay)^{q+1}).
\]
The cardinality of this set is  $|\mathcal{B'}_{q,\LL}| \le 2^{3l+1}$ whereas
$|\Pi_1| = (2^{2l}-1)(2^{2l}-2^l)/2 \approx 2^{4l-1}$. Thus $\mathcal{B'}_{q,\LL}$ is strictly 
included in $\{ (1,0,0,1)_q \} \times \Pi_1$.

Note also that the butterfly structure is defined only for odd $l$ whereas 
Proposition \ref{prop_s} covers the even case as well. We also note that 
$\Pi_1 \times \Pi_1$ and, in particular, the functions from the generalized 
butterfly construction seems to be a good source for cryptographically 
interesting functions. The engineering aspects of the butterflies were explained
in \cite{PUB}.
\end{remark}

\subsection{The even $l$ case}

By Lemma \ref{lem_sset} we can assume $f \in \Pi_1$ when $l$ is even.

\begin{proposition}\label{prop_even}
Let $q = 2^k$, $l > 3$ even, $r = 2^l$, $\Gcd(k,l) = 1$ and 
\begin{align*}
F : \LL \times \LL &\to \LL \times \LL,\\
            (x,y)  &\mapsto (f(x,y),g(x,y)),
\end{align*}
where $f \in \Pi_1$ and $g \in \Vset$. 
Then $F$ is APN if and only if 
$ F \GLEQ G_{q+1}$ or $F \GLEQ G_{q+r}$.
\end{proposition}

\begin{proof}
By Lemma \ref{lem_sset} and Proposition \ref{prop_s}, we must have
$g \in \Pi_1$. Assume $f(x_0,1) = 0 = g(x_0,1)$ for some 
$x_0 \in \ProjectiveLine(\LL)$, then since both $f,g$ have
only one $\PLL$-zero, there exists $x_1 \in \LL$ such that
$f(x_1,1)g(x_1,1) \ne 0$. Now $f(x,1)+rg(x,1) = 0$ has 
at least two $\PLL$-zeroes $\{x_0,x_1\}$ where the nonzero
$r$ is chosen to satisfy
$
r = \frac{f(x_1,1)}{g(x_1,1)}.
$
Thus $(f,f+rg) \GLEQ (f,g)$ is not APN by Proposition \ref{prop_s}.

Therefore we can assume that $f$ and $g$ do not have a common zero.
We must have $rf(x,1)+sg(x,1) \in \Pi_1$ for every 
$(r,s) \in \LL \times \LL \setminus \{(0,0)\}$.
That is to say
\[
\pi(x) = \frac{f(x,1)}{g(x,1)} = \frac{s}{r}
\]
has a unique solution $x \in \PLL$ for every $s/r \in \PLL$, 
i.e., $x \mapsto \pi(x)$ is bijective.
That is to say 
\[
F \GLEQ G_{q+1} \textrm{ or } F \GLEQ G_{q+r}, 
\]
by Lemma \ref{lem_fracproj}.

\end{proof}

\subsection{The odd $l$ case}

By Lemma \ref{lem_sset}, we can assume that $f \in \{(0,1,1,0)_q\} \cup \Pi_0 $ when $l$ is odd.

\begin{proposition}\label{prop_odd}
Let $q = 2^k$, $l > 3$ odd, $r = 2^l$, $\Gcd(k,l) = 1$ and 
\begin{align*}
F : \LL \times \LL &\to \LL \times \LL,\\
            (x,y)  &\mapsto (f(x,y),g(x,y)),
\end{align*}
where $f \in \{(0,1,1,0)_q\} \cup \Pi_0$ and $g \in \Vset$. 
Then $F$ is APN if and only if 
\begin{itemize}
\item $F \GLEQ G_{q+1}$ if $k$ is odd, and
\item $F \GLEQ G_{q+r}$ if $k$ is even.
\end{itemize}
\end{proposition}

\begin{proof}
It is clear by Proposition \ref{prop_s} that $f,g \in \Pi_3 \cup \Pi_0$.
Now, if $f,g \in \Pi_0$ then $rf+sg \in \Pi_0$ for all 
$(r,s) \in \LL \times \LL \setminus \{(0,0)\}$  is absurd.
Thus, whenever $rf+sg$ has a $\PLL$-solution it must have exactly
three $\PLL$ solutions again by Proposition \ref{prop_s}. Since, 
$(f,g) \GLEQ (rf+sg,g)$, we can assume that $f = (0,1,1,0)_q$ and $g \in  \Pi_0 \cup \Pi_3$.
Let $g = (a,b,c,d)_q$. We can assume by the left action of $\GL(2,\LL)$, $b = 0$, and 
by scaling, $a = 1$ (note that $a = 0$ is impossible since it implies $g \in \Pi_2$). 
Thus $g = (1,0,c,d)_q$. We have
\begin{align*}
 \DE{f}{\infty}(x,y) &= y^q + y = 0,\\
 \DE{g}{\infty}(x,y) &= (x^q + x) + c y^q = 0,
\end{align*}
which always have $(x,y) \in \f{2} \times \{0\}$ as solutions. Thus, 
$\trace{c} = 1$, so that $\f{2} \times \{1\}$ do not
give more solutions, in particular $c \ne 0$. Also
\begin{align*}
 \DE{f}{0}(x,y) &= x^q + x = 0,\\
 \DE{g}{0}(x,y) &= c x + d (y^q + y) = 0,
\end{align*}
implying $d \ne 0$.

Now we will introduce the main argument: If $F = (f,g)$ is APN then 
$rf+g \in \Pi_0 \cup \Pi_3$. Let us determine the conditions on $F$ when
$rf+g \in \Pi_1$. We will show that this is always the case unless $F$
is equivalent to a Gold map.

By Lemma \ref{lem_0123} (iii), $[(1,0,0,1)_q]_{\Peq} = \Pi_1$.
Let
\[
A_{\alpha,\beta,u,v} = \begin{pmatrix}
\alpha   & u\alpha \\
\beta    & v\beta
\end{pmatrix}
\]
where $\alpha,\beta \in \LL \setminus \f{2}$ and $u,v \in \LL$ with $u \ne v$
so that $\det (A_{\alpha,\beta,u,v}) = \alpha\beta(u+v) \ne 0$.

Let $h = (1,0,0,1)_q$ and consider the right action of $\PGL(2,\LL)$ 
(viewing $h$ as a univariate $q$-projective polynomial), 
\[
	h'(x) = \beta^{q+1}(x + v)^{q+1}h(\mu_{\alpha,\beta,u,v}(x)),
\]
where 
\[
\mu_{\alpha,\beta,u,v}(x) : x \mapsto \frac{\alpha (x + u)}{\beta (x + v)}.
\]
Equating the terms of $h'$ and $rf+g$, we get
\begin{align*}
\alpha^{q+1}         + \beta^{q+1}         &= 1,\\
\alpha^{q+1} u       + \beta^{q+1} v       &= r,\\
\alpha^{q+1} u^q     + \beta^{q+1} v^q     &= r+c,\\
\alpha^{q+1} u^{q+1} + \beta^{q+1} v^{q+1} &= d.
\end{align*}
Since $x \mapsto x^{q+1}$ is bijective and $r \in \LL$, we can rewrite these as
\begin{align*}
\gamma (u^q + u) + (\gamma+1) (v^q + v) &= c,\\
\gamma u^{q+1}   + (\gamma+1) v^{q+1}   &= d,
\end{align*}
where $\gamma = \alpha^{q+1} = 1 + \beta^{q+1}$. Since $\alpha,\beta \in \LL \setminus \f{2}$,
so is $\gamma$ and for all $\gamma \in \LL \setminus \f{2}$, we can find $\alpha$ and $\beta$
satisfying the above equalities. Equivalently,
\begin{align}
\label{eq1} \gamma (u^q + u + c) &= (\gamma+1) (v^q + v + c),\\
\label{eq2} \gamma (u^{q+1} + d) &= (\gamma+1) (v^{q+1} + d).
\end{align}
Note that $\trace{c} = 1$ and $\gamma \not\in \f{2}$, therefore neither side of
\eqref{eq1} vanishes. Thus we can divide 
\eqref{eq2} by \eqref{eq1} side by side to get
\begin{equation}
\label{eq3} \varphi_{c,d}(u) = \varphi_{c,d}(v),
\end{equation}
where
\[
\varphi_{c,d} : x \mapsto \frac{x^{q+1}+d}{x^q+x+c}. 
\]
Clearly \eqref{eq1} and \eqref{eq2} hold together if and only if 
        \eqref{eq1} and \eqref{eq3} hold together.

Now assume $\varphi_{c,d}$ is not bijective on $\LL$. Then there exist $u,v \in \LL$ such that 
$\varphi_{c,d}(u) = \varphi_{c,d}(v)$ with $u + v \not\in \f{2}$. It is clear that
\[
\varphi_{c,d}(x) + \varphi_{c,d}(x+1) =  \frac{x^{q+1}+d+(x+1)^{q+1}+d}{x^q+x+c} = \frac{x^q+x+1}{x^q+x+c} \ne 0,
\]
since $\trace{1}=1$. For such $u,v$, \eqref{eq1} becomes
\[
\frac{u^q + u + c}{v^q + v + c} = \frac{\gamma+1}{\gamma} = 1 + \frac{1}{\gamma} \not\in \f{2}.
\]
Thus, selecting $\alpha,\beta \in \LL \setminus \f{2}$ we can produce such $\gamma$. Hence, $\varphi_{c,d}$ 
must be bijective, and by Lemma \ref{lem_fracproj}, we must have $c = d = 1$ and $g = (1,0,1,1)_q$. 

Now let $\langle \xi \rangle =  \f{4}^\times$. Any $X \in \f{2^{2l}}$ can be written 
as $X = x+y\xi$ where $x,y \in \LL$. We have
\[
\xi^{2^k} = \left\{
\begin{array}{ll}
\xi + 1 & \textrm{if $k$ is odd,}\\
\xi     & \textrm{if $k$ is even}.
\end{array}
\right.
\]
Now if $k$ is odd, then
\begin{align*}
	(x+y\xi)^{q+1} &= x^{q+1}+ xy^q +y^{q+1}+\xi (x^qy + xy^q)\\
	               &\cong ((x^{q+1}+xy^q+y^{q+1}),x^qy+xy^q)\\
	               &\GLEQ ((0,1,1,0)_q,(1,0,1,1)_q),
\end{align*}
and if $k$ is even,
\begin{align*}
	(x+y\xi)^{q+r} &= x^{q+1}+ x^qy +y^{q+1}+\xi (x^qy + xy^q)\\
	               &\cong ((x^{q+1}+x^qy+y^{q+1}),x^qy+xy^q)\\
	               &\GLEQ ((x^{q+1}+xy^q+y^{q+1}),x^qy+xy^q)\\
	               &\GLEQ ((0,1,1,0)_q,(1,0,1,1)_q),
\end{align*}
using $(x,y) \mapsto (y,x)$ in the penultimate line,
proving our assertion that $(f,g) \GLEQ G_{q+1}$ when $k$ is odd,
and                        $(f,g) \GLEQ G_{q+r}$ when $k$ is even.
Note that when $l$ is odd and $\Gcd(k,l) = 1$, we have 
$\Gcd(k,2l)   = 1$ if $k$ is odd and $\Gcd(l-k,2l) = 1$ if $k$ is even.
By \cite{gold}, these maps are APN since $q+r = q(1+r/q) = 2^k(1+2^{l-k})$. 
\end{proof}

\subsection{Proof of Theorem \ref{thm_main}}

Now, Propositions \ref{prop_gcd}, \ref{prop_s}, \ref{prop_even} and \ref{prop_odd}
together with Lemma \ref{lem_sset} proves Theorem \ref{thm_main}.

\section{Acknowledgments}
The author would like to thank Claude Carlet for his comments.
He also thanks Lukas K\"olsch for many discussions that improved 
the treatment considerably. 

This work was supported by the {\sf GA\v{C}R Grant 18-19087S - 301-13/201843}.

\bibliographystyle{amsplain}
\bibliography{biproj_arxiv}

\end{document}